\documentclass[11pt]{amsart}
\usepackage[numbered]{bookmark}
\usepackage{changepage} 
\usepackage[utf8]{inputenc}
\usepackage{graphics, graphicx}
\usepackage{amsmath, amsfonts, amsthm, amssymb}
\usepackage{mathrsfs, mathtools, mathabx}
\usepackage{verbatim} 
\usepackage{enumerate}
\usepackage{xcolor}
\usepackage{cite}

\usepackage[T1]{fontenc}

\newtheorem{theorem}{Theorem}[section]

\newtheorem{lemma}[theorem]{Lemma}
\newtheorem{proposition}[theorem]{Proposition}
\newtheorem{corollary}[theorem]{Corollary}

\newtheorem*{claim}{Claim}

\newcommand{\Z}{\mathbb{Z}} 					
\newcommand{\N}{\mathbb{N}}					
\newcommand{\Q}{\mathbb{Q}}					
\newcommand{\R}{\mathbb{R}}					
\newcommand{\T}{\mathbb{T}}					

\newcommand{\A}{\mathcal{A}}					
\newcommand{\Dom}{X_d}						
\newcommand{\cI}[1]{\widehat{#1}}				
\newcommand{\cR}{\mathcal{R}}					
\newcommand{\cZ}{\mathcal{Z}}					
\newcommand{\cZE}{{\widehat{\cZ}}}				

\newcommand{\Fd}[1]{\{1, \dots, #1\}}				
\newcommand{\twr}{\mathcal{T}}

\newcommand{\be}{\begin{equation}}
\newcommand{\ee}{\end{equation}}

\newcommand{\Function}[5]{\begin{array}{cccc} #1 : & #2 & \rightarrow & #3 \\ & #4 & \mapsto & #5 \end{array}}

\newcommand{\orth}[1]{#1^\bot}

\date{}
\begin{document}
\author{Frank Trujillo}
\address{Institut für Mathematik, Universität Zürich, Winterthurerstrasse 190, CH-8057 Zürich, Switzerland}
\email{frank.trujillo@math.uzh.ch}
\title[On the HD of inv. measures of piecewise smooth circle maps]{On the Hausdorff dimension of invariant measures of piecewise smooth circle homeomorphisms}
\begin{abstract}
We show that, generically, the unique invariant measure of a sufficiently regular piecewise smooth circle homeomorphism with irrational rotation number and zero mean nonlinearity (e.g., piecewise linear) has zero Hausdorff dimension. 

To encode this generic condition, we consider piecewise smooth homeomorphisms as \emph{generalized interval exchange transformations} (GIETs) of the interval and rely on the notion of \emph{combinatorial rotation number} for GIETs, which can be seen as an extension of the classical notion of rotation number for circle homeomorphisms to the GIET setting.
\end{abstract}
\maketitle

\section{Introduction}

Recall that an \emph{irrational circle homeomorphism} $f: \T \to \T$, that is, a continuous bijection on $\T = \R / \Z$ with no periodic points, is uniquely ergodic \cite{furstenberg_strict_1961}. Moreover, this map is topologically conjugated to an irrational rotation $x \mapsto x + \alpha$ on $\T$ if and only if it does not admit wandering intervals, the latter condition being guaranteed if, for example, the map is piecewise smooth and its derivative has bounded variation. In this case, its unique invariant probability measure $\mu_f$ can be expressed as the pushforward $\mu_f = h_\ast \textup{Leb}$ of the Lebesgue measure on $\T$ by the conjugacy map $h :\T \to \T$. 

The aim of many recent works has been to understand more deeply this unique invariant probability measure under different assumptions on the map. Let us point out that fine statistical properties of $\mu_f$ are closely related to the (lack of) regularity of $h$, but, in general, it is not possible to study the conjugacy map directly in order to understand dimensional properties of the invariant measure. Nevertheless, the existence of this conjugacy allows the transport of valuable techniques, such as \emph{dynamical partitions} (see Section \ref{sc: dynamical_partitions}), from rigid rotations to more general maps, and to use these tools to study geometric properties of the associated invariant measures. 

Of particular interest to us will be the notion of Hausdorff dimension of a probability measure $\mu$ on $\T$, defined as \[\dim_H(\mu) := \inf \{ \dim_H(X) \mid \mu(X) = 1 \}, \] where $\dim_H(X)$ denotes the Hausdorff dimension of the set $X$. Intuitively, this is a way to assess the `size' of the set where the measure is concentrated whenever the measure is singular with respect to Lebesgue. 

For circle diffeomorphisms, it follows from the works of M. Herman \cite{herman_sur_1979} and J.C. Yoccoz \cite{yoccoz_conjugaison_1984} that sufficiently regular circle diffeomorphisms are smoothly conjugated to a rigid rotation provided its rotation number is Diophantine. Hence, for any smooth circle diffeomorphism with a Diophantine rotation number, its unique invariant measure is equivalent to the Lebesgue measure, and therefore, its Hausdorff dimension is equal to one. On the other hand, for any $0 \leq \beta \leq 1$ and any \textit{Liouville} number $\alpha$, that is, any non-Diophantine irrational number, V. Sadovskaya \cite{sadovskaya_dimensional_2009} has constructed examples of smooth diffeomorphisms with rotation number $\alpha$ whose unique invariant measure has Hausdorff dimension $\beta$. In the analytic category, V. Arnold showed the existence of analytic circle maps with Liouville rotation number whose conjugacy to the circle rotation is non-differentiable.

As for \emph{critical circle maps with power-law criticalities}, that is, smooth diffeomorphisms with a finite number of singular points where the derivative vanishes and where the map behaves as $x \mapsto x|x|^p + c$ in the neighborhood of the critical point, for some $p > 0$ and $c \in \R$ which may depend on the critical point, K. Khanin \cite{khanin_universal_1991} proved that the unique invariant measure of any sufficiently regular irrational critical circle map is singular with respect to the Lebesgue measure. If, in addition, the rotation number of the map is of bounded type, J. Graczyk and G. Świątek \cite{graczyk_singular_1993} showed that the Hausdorff dimension of the unique invariant measure is bounded away from $0$ and $1$. More recently, the author \cite{trujillo_hausdorff_2020} provided explicit bounds, depending only on the arithmetic properties of the rotation number, for the Hausdorff dimension of these maps.

In this work, we study the \emph{Hausdorff dimension} of the unique invariant probability measures of certain piecewise smooth circle homeomorphisms, known as \emph{$P$-homeomorphisms} or \emph{circle diffeomorphisms with breaks}. 
 We restrict ourselves to the setting of maps with \emph{zero mean nonlinearity}, as the non-zero case has been recently treated by K. Khanin and S. Kocić in \cite{khanin_hausdorff_2017}. We refer the reader to Section \ref{sc: P-homeomorphisms} for the relevant definitions. 
 
 We aim to show that for `most' of these maps, the associated unique invariant probability measure has zero Hausdorff dimension. For clarity of exposition, we postpone the exact statement of this result to Section \ref{sc: results}. However, let us mention that to give sense to the word `most' in the previous statement, we rely on the notion of \emph{combinatorial rotation number} (see Section \ref{sc: combinatorial_rotation_number}), which is widely used in the theory of interval exchange transformations and can be seen as an extension of the classical notion of rotation number for circle maps. 
 
Recall that a homeomorphism of the circle $f: \T \to \T$ is called a \emph{P-homeomorphism} or a \emph{circle diffeomorphisms with breaks} if it is a smooth orientation preserving homeomorphism, differentiable away from countable many points, so-called \emph{break-points}, at which left and right derivatives, denoted by $Df_-$, $Df_+$ respectively, exist but do not coincide, and such that $\log Df$ has bounded variation. 
A P-homeomorphism that is linear in each domain of differentiability is called \emph{PL-homeomorphism}. 

Let us mention that due to a recent result of P. Berk and the author \cite{berk_rigidity_2022}, which shows that a typical sufficiently smooth P-homeomorphism without periodic points and zero mean nonlinearity is smoothly conjugated to a PL-homeomorphism, to prove our main result (Theorem \ref{thm: HD_multibreak}) it suffices to consider PL-homeomorphisms.

The invariant measures of PL-homeomorphisms were first studied by M. Herman \cite{herman_sur_1979}, who showed that a PL-homeomorphism with exactly two break points and irrational rotation number has an invariant measure absolutely continuous with respect to Lebesgue if and only if its break points lie on the same orbit. 

More generally, I. Liousse \cite{liousse_nombre_2005} showed that the invariant measure of a generic PL-homeomorphism with a finite number of break points and irrational rotation number of bounded type is singular with respect to Lebesgue. The generic condition in \cite{liousse_nombre_2005} is explicit and appears as an arithmetic condition on the logarithm of the slopes of the PL-homeomorphism. 

For general P-homeomorphisms with exactly one break point and irrational rotation number, A. Dzhalilov and K. Khanin \cite{dzhalilov_invariant_1998} showed that the associated invariant probability measure is singular with respect to Lebesgue. The case of two break points has been studied by A. Dzahlilov, I. Liousse \cite{dzhalilov_circle_2006} in the bounded rotation number case, and by A. Dzahlilov, I. Liousse and D. Mayer \cite{dzhalilov_singular_2009} for arbitrary irrational rotation numbers. In both works, the authors conclude the singularity of the associated invariant probability measure. 

More recently, for P-homeomorphisms of class $C^{2 + \epsilon}$ with a finite number of break points and nonzero mean nonlinearity, K. Khanin and S. Kocić \cite{khanin_hausdorff_2017} showed that the Hausdorff dimension of their unique invariant measure is equal to $0$, provided that their rotation number belongs to a specific (explicit) full-measure set of irrational numbers. In the same work, the authors show that this result cannot be extended to all irrational rotation numbers. 

Recall that the \emph{mean nonlinearity} of a piecewise $C^2$ circle map such that $D\log Df \in L^1$ is given by $\mathcal{N}(f) = \int_{\T} D\log Df(x) dx,$  where, as an abuse of notation, for any $k \geq 1$, we use $D^kf$ to denote the $k$-th derivative of any lift $F: \R \to \R$ of $f$.

In the following, we will introduce the core notions used throughout this work and briefly overview various results in different settings. We will focus our discussion on the piecewise smooth case, which will be this work's main object of study. We will finish this introduction with a brief discussion of the difficulties of extending the techniques used in previous works and by providing precise statements of our main results.

\section*{Acknowledgements} 
I thank Corinna Ulcigrai for her constant support and fruitful discussions during the preparation of this work. The author was supported by the  {\it Swiss National Science Foundation} through Grant $200021\_188617/1$.

\section{Preliminaries}
\subsection{Circle maps}

Let us quickly recall some of the properties of circle maps that will be used throughout this work.

\subsubsection{Rotation number}

Let $f : \T \rightarrow \T$ be an orientation-preserving circle homeomorphism and let $F : \R \rightarrow \R$ be a \emph{lift} of $f$, that is, a continuous homeomorphism of $\R$ such that $F(x+1) = F(x) + 1$ and $F (x) (\bmod$ $1) = f(x)$ for all $x \in \R$. By a classical result of Poincaré, the limit 
\[ \rho(f) = \lim_{n \to \infty} \dfrac{F^n(x)}{n} \mod 1,\]
is well-defined and independent of the value $x \in \R$ initially chosen. This limit is called the \textit{rotation number} of $f$. By the Poincaré's classification theorem for circle maps, any transitive orientation preserving circle homeomorphism with irrational rotation number $\alpha \in \R \setminus \Q$ is conjugate to the \emph{rigid rotation} $R_\alpha: \T \to \T$ given by $R_\alpha(x) = x + \alpha$, for any $x \in \T$. In particular, such a map is uniquely ergodic. 

\subsubsection{Dynamical partitions}
\label{sc: dynamical_partitions}

Given a circle homeomorphism $f: \T \to \T$ topologically conjugated to an irrational rotation $R_\alpha$, a classical way to study fine statistical properties of its unique invariant measure $\mu_f$ is to consider the so called \emph{dynamical partitions} of $f$, which are defined through the denominators $(q_n)_{n \geq 0}$ of the \emph{convergents}
\[ \dfrac{p_n}{q_n} = [a_1,a_2,a_3,\dots,a_n],\]
associated to the \emph{continued fraction} of $\alpha$
\[ \alpha = \cfrac{1}{a_1 + \cfrac{1}{a_2+ \cfrac{1}{a_3 +\cfrac{1}{\cdots}}}} = [a_1,a_2,a_3,\dots],\] 
as follows: Given $x_0 \in \T$ and $n \geq 1$, the $n$-th \emph{dynamical partition} $\mathcal{P}_n(x_0)$ of $f$, with base point $x_0$, is given by
\[ \mathcal{P}_n(x_0) = \lbrace I_{n - 1}^0(x_0), \dots ,I_{n - 1}^{q_n - 1}(x_0) \rbrace \cup \lbrace I_n^0(x_0),\dots ,I_n^{q_{n-1}-1}(x_0) \rbrace,\]
where $I_m^i(x_0) = f^i(I_m(x_0))$ is the $i$-th iterate of the circle arc given by
\begin{equation}
\label{eq: fundamental_interval}
I_m(x_0) = \left\{ \begin{array}{lcl} [x_0, f^{q_m}(x_0)) & \text{ if }& m \text{ is even}, \\ \left[f^{q_m}(x_0), x_0\right) & \text{ if }& m \text{ is odd}.\end{array} \right.
\end{equation}
These partitions form a refining sequence. In fact, it is easy to see that $I_{n + 1}(x_0)\subset I_{n - 1}(x_0),$ for any $n \geq 1$. Moreover, they verify
 \begin{equation}
\label{eq: refining_property}
 I_{n - 1}^i(x_0)\setminus I_{n + 1}^i(x_0)= \bigcup_{j = 0}^{a_{n + 1} - 1} I_{n}^{i + q_{n - 1} + jq_{n}}(x_0),
 \end{equation}
 for all $0 \leq i < q_n$ and all $n \geq 1.$ Notice that the RHS of (\ref{eq: refining_property}) is a disjoint union of $a_{n+1}$ different iterates of $I_n$. Furthermore, each iterate is \emph{adjacent} to the next one when seen as arcs in the circle; that is, they share a common endpoint.
 
Geometric properties of these partitions are closely related to dimensional properties of the subjacent invariant measure, see for example \cite{khanin_universal_1991}, \cite{khanin_hausdorff_2017}, \cite{trujillo_hausdorff_2020}.

\subsubsection{Renormalization}

The renormalization maps of $f$ are closely related to the dynamical partitions and are often used to deduce geometric properties of these partitions. The \emph{$n$-th renormalization of $f$}, with base point $x_0$, is the map $f_n: I_{n - 1}(x_0) \cup I_n(x_0) \to I_{n - 1}(x_0) \cup I_n(x_0)$ given by
\begin{equation}
\label{eq: renormalization}
 f_n(x) = \left\{ \begin{array}{lcl} f^{q_{n-1}}(x) & \text{ if }& x \in I_n(x_0), \\ f^{q_n}(x) & \text{ if }& x \in I_{n - 1}(x_0). \end{array} \right.
\end{equation}
It is not difficult to check that the map $f_n$ is the first return map of $f$ to $I_{n - 1}(x_0) \cup I_n(x_0).$

\subsection{Interval exchange transformations.}

Let $I=[0,1)$ be the unit interval. An \emph{interval exchange transformation} (IET) is a bijective, right-continuous function $T\colon I \to I$, with a finite number of discontinuities, whose restriction to any subinterval of continuity is given by a translation. This means that the interval $I$ admits a decomposition into a finite number of subintervals which $T$ permutes. For notational simplicity, we denote this partition by $\{ I_\alpha\}_{\alpha \in \A}$, where the indexes belong to some finite alphabet $\A$ with $d \geq 2$ symbols.

An IET $T$ of $d$ subintervals can be encoded by a pair $(\lambda, \pi)$ corresponding to a \emph{combinatorial datum} $\pi = (\pi_t, \pi_b)$, consisting of two bijections $\pi_t, \pi_b : \A \rightarrow \Fd{d}$ describing the order of the intervals before and after $T$ is applied (the letters \emph{t, b} stand for \emph{top} and \emph{bottom}), and a \emph{lengths vector} $\lambda = (\lambda_\alpha)_{\alpha \in \A}$ in the simplex $\Delta_{d} = \{ {\nu}\in \mathbb{R}_+^\A \mid \sum_{\alpha\in\A}\nu_\alpha=1\}$, which corresponds to the lengths of the intervals in the partition $\{ I_\alpha\}_{\alpha \in \A}$ associated to $T.$ We call $\pi_b \circ \pi_t^{-1} : \Fd{d} \to \Fd{d}$ the \emph{monodromy invariant} of $\pi$.

A combinatorial datum $\pi = (\pi_t, \pi_b)$ is said to be of \emph{rotation type} if its monodromy invariant verifies
\[ \pi_b \circ \pi_t^{-1}(i) - 1 = i + k \textup{ (mod } d),\]
for some $k \in \{0, \dots, d - 1\}$ and all $i \in \Fd{d}$. Similarly, we say that an IET is of \emph{rotation type} if its combinatorial datum is of rotation type. 

Given an IET $T$ with associated partition $\{ I_\alpha\}_{\alpha \in \A}$, we can obtain an explicit expression for the intervals $I_\alpha$ as $[l_\alpha, r_\alpha)$ where $$l_\alpha = \sum_{\pi_t(\beta) < \pi_t(\alpha)} \lambda_\beta, \hskip1cm r_\alpha = l_\alpha + \lambda_\alpha.$$ Notice that $\{l_\alpha\}_{\pi_t(\alpha) \neq 1}$ are the only possible discontinuity points of $T$. With these notations
\[ T(x) = x + w_\alpha,\]
for any $x \in I_\alpha$ and any $\alpha \in \A$, where
\[ w_\alpha = \sum_{\pi_b(\beta) < \pi_b(\alpha)} \lambda_\beta - \sum_{\pi_t(\beta) < \pi_t(\alpha)} \lambda_\beta.\]
We denote by $w = (w_\alpha)_{\alpha \in \A}$ the \textit{translation vector} of the IET $T$. Notice that $w_\alpha$ can be expressed as a linear transformation on $\R^\A$ as $w_\alpha = \Omega_\pi(\lambda)$, where $\Omega_\pi : \R^A \rightarrow \R^A$ is given by 
\begin{equation}
\label{eq: intersection_matrix}
\Omega_{\alpha, \beta} = \left\{ \begin{array}{cl} +1 & \text{if } \pi_b(\alpha) > \pi_b(\beta) \text{ and } \pi_t(\alpha) < \pi_t(\beta), \\
-1 & \text{if } \pi_b(\alpha) < \pi_b(\beta) \text{ and } \pi_t(\alpha) > \pi_t(\beta), \\
0 & \text{in other cases.} 
\end{array}\right.
\end{equation}
We say that the pair $\pi = (\pi_0, \pi_1)$ is \textit{reducible} if there exists $1 \leq k < d$ such that 
\[ \pi_1 \circ \pi_0^{-1} (\Fd{k}) = \Fd{k}.\]
Otherwise, it is said to be \textit{irreducible}. An IET is said to be reducible (resp. irreducible) if the associated combinatorial datum $\pi$ is reducible (resp. irreducible). We say that an IET $T$ satisfies the \emph{Keane condition} if the orbits of the possible discontinuities $l_\alpha$, with $\alpha \in \A$ and such that $\pi_t(\alpha)\neq 1$, are infinite and disjoint. In particular, any IET verifying the previous condition is irreducible. Recall that a transformation on a metric space is said to be \emph{minimal} if the orbit of all points is dense. By \cite{keane_interval_1975}, any IET satisfying Keane's condition is minimal.

\subsubsection{Rauzy-Veech renormalization}
\label{sc: RV}

Let $T$ be an IET with $d$ subintervals encoded by $(\lambda, \pi)$ and such that $\lambda_{\alpha_t} \neq \lambda_{\alpha_b}$. Denote
\[\alpha_s = \pi_s^{-1}(d),\] for $s = t, b$. The letters $\alpha_t$ and $\alpha_b$ correspond to the `last' intervals (that is, those having $1$ as their right endpoint) in the partitions $\{ I_\alpha\}_{\alpha \in \A}$ and $\{ T(I_\alpha)\}_{\alpha \in \A}$, respectively. By comparing the lengths of these intervals, we define the \emph{type} of $T$ as
\[ \epsilon(\lambda, \pi) = \left\{ \begin{array}{cl}
0 & \text{ if } \lambda_{\alpha_t} > \lambda_{\alpha_b}, \\ 
1 &\text{ if }\lambda_{\alpha_t} < \lambda_{\alpha_b}.
 \end{array}\right.\]
The longest of these two intervals is sometimes referred to as the \textit{winner} and the shortest as the \textit{loser}. Notice that $\alpha_{\epsilon(\lambda, \pi)}$ and $\alpha_{1 - \epsilon(\lambda, \pi)}$ correspond to the symbols of the winner and the loser intervals respectively. If there is no risk of confusion, we will denote these letters simply by $\alpha_\epsilon$ and $\alpha_{1 - \epsilon}$. We will sometimes refer to types $0$ and $1$ as \emph{top} and \emph{bottom} respectively. 

The \emph{Rauzy-Veech induction} of $T$, which we denote by $\cI{T}$, is defined as the first return map of $T$ to the subinterval 
\[ \cI{I} = \left\{ \begin{array}{cl}
I \,\setminus\, T\left(I_{\alpha_b}\right) & \text{ if } T \text{ is of type top}, \\ 
I \,\setminus \,I_{\alpha_t} &\text{ if } T \text{ is of type bottom}.
 \end{array}\right.\]
We define the \emph{Rauzy-Veech renormalization} of $T$, which we denote by $\cR(T)$, by rescaling linearly $\cI{T}$ to the interval $I$. This renormalization will be an IET with the same number of subintervals as $T$. Furthermore, it is not difficult to check that this process can be iterated infinitely many times if and only if $T$ verifies Keane's condition. Also, it follows quickly from the definition that each infinitely renormalizable pair $(\lambda, \pi)$ will admit exactly two preimages. We refer the interested reader to \cite{viana_ergodic_2006} for proof of these facts.

Denote by $\mathfrak{G}_d$ the \emph{set of combinatorial data} $\pi = (\pi_t, \pi_b)$ with $d$ symbols and let $\mathfrak{G}_d^0 \subset \mathfrak{G}_d$ be the subset of \emph{irreducible combinatorial data}, which we equip with the counting probability measure $d\pi$. We will sometimes refer to a combinatorial datum in $\mathfrak{G}_d$ simply as a \emph{permutation}. 

Let us point out that for a fixed $\pi \in \mathfrak{G}_d^0$ and for any $\lambda \in \Delta_d$ such that $(\lambda, \pi)$ is renormalizable, there are only two possibilities for the combinatorial datum of $\cR(\lambda, \pi)$, depending on whether the IET is of top or bottom type. For notational simplicity, for any $\pi \in \mathfrak{G}^0_d$ and for any $\epsilon \in \{0, 1\}$, we let $\Delta_\pi = X_d \cap (\Delta_d \times \{\pi\})$ and denote by $\Delta_{\pi, \epsilon}$ the set of of IETs in $\Delta_\pi$ of type $\epsilon$. 

Notice that the set $\Dom$ of pairs $(\lambda, \pi)$ verifying Keane's condition is $\cR$-invariant and of full measure in $\Delta_d \times \mathfrak{G}_d^0$, with respect to $\textup{Leb} \times d\pi$. Hence $$\cR: \Dom \rightarrow \Dom$$ is a well defined 2 to 1 map defined in a full measure subset $\Dom \subset \Delta_d \times \mathfrak{G}_d^0$. In fact, it is easy to see that any $(\lambda, \pi) \in X_d$ has exactly two preimages, one of type top and one of type bottom. Moreover, for any $\pi \in \mathfrak{G}^0_d$ and for any $\epsilon \in \{0, 1\}$, the map $\cR\mid_{\Delta_{\pi, \epsilon}}: \Delta_{\pi, \epsilon} \to \Delta_{\pi(\epsilon)}$ is bijective, where we denote by $\pi(\epsilon)$ the permutation obtained from $\pi$ after a Rauzy-Veech renormalization of type $\epsilon$.

 It was shown independently by H. Masur \cite{masur_interval_1982} and W. Veech \cite{veech_gauss_1982} that $\cR$ it admits an infinite ergodic invariant measure $\mu_\cR$, absolutely continuous with respect to $\textup{Leb} \times d\pi$. Moreover, the measure $\mu_\cR$ is unique up to product by a scalar.

\subsubsection{Rauzy classes}

Given $\pi, \pi' \in \mathfrak{G}_d$, the permutation $\pi'$ is said to be a \textit{sucessor} of $\pi$ if there exist $\lambda, \lambda' \in \Delta_d$ such that $(\pi', \lambda') = \cR(\pi, \lambda).$ We denote this relation by $\pi \rightarrow \pi'$. Notice that any successor of an irreducible permutation is also irreducible. 

The relation `$\rightarrow$' defines an oriented graph structure on the set of irreducible permutations $\mathfrak{G}^0_d$. We call \textit{Rauzy classes} the connected components of the oriented graph $\mathfrak{G}^0_d$ with respect to the successor relation. 

Let us point out that if $\pi, \pi'$ belong to the same Rauzy class, then an oriented path exists in $\mathfrak{G}^0_d$ from $\pi$ to $\pi'$. For combinatorial data of rotation type, we have the following.

\begin{proposition}
For any $d \geq 2$, the permutations of rotation type belong to the same Rauzy class in $\mathfrak{G}^0_d$.
\end{proposition}

\subsubsection{GIETs and combinatorial rotation number}
\label{sc: combinatorial_rotation_number}

A \emph{generalized interval exchange transformation} (GIET) is a bijective, right-continuous function $T\colon I \to I$ with a finite number of discontinuities, such that its restriction to any subinterval of continuity is given by a smooth function whose derivative is non-negative and extends to the closure of the subinterval. 

Rauzy-Veech renormalization, initially defined only for IETs, extends trivially to the context of GIETs. An infinitely renormalizable GIET $f$ defines a unique path $\gamma(f)$ on the Rauzy diagram, which we call the \emph{combinatorial rotation number} or simply the \emph{rotation number} of $f$. The notion of rotation number for GIETs is now classical and goes back to the works of S. Marmi, P. Moussa and J. C. Yoccoz \cite{marmi_affine_2010}, \cite{marmi_linearization_2012}.

An infinite path on the Rauzy diagram is called \emph{$\infty$-complete} if each letter in $\A$ wins infinitely many times. We say that a GIET is \emph{irrational} if it is infinitely renormalizable and if its rotation number is $\infty$-complete. 

 Let us point out that in the case of GIETs, the role played by IETs and combinatorial rotation numbers are analogous to that of rigid rotations and rotation numbers in the case of circle homeomorphisms. Indeed, an infinite path in a Rauzy diagram is associated with some infinitely renormalizable IET if and only if it is $\infty$-complete. Also, two infinitely renormalizable IETs are conjugated if and only if they have the same rotation number. Moreover, any irrational GIET is semi-conjugated to a unique IET with the same rotation number. Analogously to the circle case, this semi-conjugacy is a conjugacy if the GIET does not admit wandering intervals. We refer the interested reader to \cite{yoccoz_echanges_2005} for proof of these facts.

Given the previous discussion, we can think of the combinatorial rotation number for irrational GIETs on $d$ intervals as taking values on $\Delta_d \times \mathfrak{G}_d^0.$ This will allow us to speak in the following of \emph{almost every rotation number} for GIETs. As an abuse of notation, we will sometimes write $\gamma(T) = (\lambda, \pi)$ to state that a GIET $T$ and a standard IET $T_0$ associated to some $(\lambda, \pi) \in X_d$ have the same rotation number. 

\subsubsection{The lengths cocycle}
\label{sc: lengths_cocycle}

Let $T = (\lambda, \pi)$ verifying Keane's condition and define 
\begin{equation} 
\label{eq: RVmatrix}
A(T) = I_{d \times d} + E_{\alpha_\epsilon, \alpha_{1 - \epsilon}},
\end{equation}
where $I_{d \times d}$ denotes the identity matrix and $E_{i, j}$ is the matrix whose entries are $1$ at the position $(i, j)$ and $0$ otherwise. Notice that the matrix $A(T)$ depends only on the combinatorial datum $\pi$ and on the type $\epsilon(\lambda, \pi).$ Denote
\[ A^n(T) = A(T) \dots A({\cR^n(T)}).\] Then the lengths vector of $\cR^n(T)$ is given by 
\[\frac{A^n(T) ^{-1}\lambda}{|A^n(T) ^{-1}\lambda|_1}.\]
The map \[\Function{A^{-1}}{\Dom}{SL(d, \Z)}{T}{A(T)^{-1}}\] is a cocycle over $\cR$, known as the \emph{Rauzy-Veech cocycle} or \emph{lengths cocycle}. 

The following observation will be of fundamental importance. For a proof see \cite[Proposition 7.6]{yoccoz_interval_2010}.

\begin{proposition}
\label{prop: trivial_action}
The set
\[\bigcup_{\pi \in \mathfrak{G}^0_d} \Delta_\pi \times \textup{Ker}(\Omega_\pi),\]
where $\Omega_\pi: \R^\A \to \R^\A$ is given by (\ref{eq: intersection_matrix}), is invariant by the lengths cocycle
\[\Function{\mathcal{L}}{X_d \times \R^n}{X_d \times \R^n}{(T, v)}{(\cR(T), A(T)^{-1}v)}\]
and its action is trivial on it. 

More precisely, one can choose a basis of row vectors of $\textup{Ker}(\Omega_\pi)$, for every $\pi \in \mathfrak{G}^0_d$, such that for any $(\lambda, \pi) \in \Delta_d \times \mathfrak{G}^0_d$ infinitely renormalizable and for any $n \in \N$, the transformation 
\[A^n(\lambda, \pi)^{-1}\mid_{\textup{Ker}(\Omega_\pi)}: \textup{Ker}(\Omega_\pi) \to \textup{Ker}(\Omega_{\pi'})\] is the identity with respect to the selected bases, where $\pi' \in \mathfrak{G}^0_d$ is such that $\mathcal{R}^n(\lambda, \pi) \in \Delta_{\pi'}$. In particular, if $\pi' = \pi$, then $A^n(\lambda, \pi)^{-1}$ acts as the identity on $\textup{Ker}(\Omega_\pi).$
\end{proposition}

\subsubsection{The heights cocycle}
\label{sc: heights_cocycle}

Let $T = (\lambda, \pi)$ verifying Keane's condition. Using the matrix $A$ in (\ref{eq: RVmatrix}), we define a cocycle over $\cR$, known as the \emph{heigths cocycle}, by 
 \[\Function{A^{T}}{\Dom}{SL(d, \Z)}{T}{A(T)^{T}}.\]
 This cocycle allows us to describe the return times, or \emph{heights} if we think of the induced map as a system of Rohlin towers associated with the iterates of $\cR$. Indeed, given $n \in \N$, the transformation $\cR^n(T)$ is defined as the linear rescaling of the first return map of $T$ to some subinterval $I^n \subset I.$ Moreover, this interval admits a decomposition $I = \bigsqcup_{\alpha \in \A} I^n_\alpha$ such that the return time to $I^n$ on each subinterval $I^n_\alpha$ is constant. Then the vector of return times to $I^n$ is given by
\[h^n = A^n(T)^T \overline{1},\]
where $\overline{1} \in \N^\A$ is the vector whose entries are all equal to $1$.

\subsubsection{Affine IETs}

An \emph{affine interval exchange transformation} (AIET) is a GIET for which the restriction to each subinterval of continuity is a linear map. As for IETs, given an AIET $f$, we can decompose the interval $I$ into intervals of continuity of $f$, which we denote by $\{ I_\alpha\}_{\alpha \in \A}$, where $\A$ is a finite alphabet. Notice that $f$ will not only change the order of these intervals but will also linearly modify its lengths. If the permutation associated with $f$ is in the Rauzy class of rotations, we say that $f$ is of \emph{rotation type}.

Given an AIET $f$ on $d$ intervals with associated partition $\{ I_\alpha\}_{\alpha \in \A},$ we define its \emph{log-slope} as the logarithm of the slope of $f$ in each interval of continuity, namely, the vector $\omega = (\log Df \mid_{I_\alpha})_{\alpha \in \A} \in \R^\A$. 

The following relation between the log-slope of Rauzy-Veech renormalizations of an AIET and the heights cocycle will be of fundamental importance. 

\begin{proposition}
\label{prop: log_slope}
Let $f$ be an irrational AIET on $d$ intervals with combinatorial rotation number $\gamma(f) = (\lambda, \pi) \in \Delta_d \times \mathfrak{G}^0_d$ and log-slope $\omega \in \R^\A$. Then the log-slope of $\cR^n(f)$ is given by $B^n(\lambda, \pi)^T \omega,$ for any $n \geq 0.$
\end{proposition}

\subsubsection{Zorich acceleration}
\label{sc: Zorich_map}

A. Zorich \cite{zorich_finite_1996} showed that one can `accelerate the dynamics' of $\cR$ to define a map $\cZ: \Dom \to \Dom$ admitting a unique ergodic invariant probability measure $\mu_{\cZ}$ which is absolutely continuous with respect to the $\textup{Leb} \times d\pi$ and whose density is a rational function on $\Delta_d$ uniformly bounded away from the boundary of $\Delta_d$. This map is given by
\[\cZ(T) = \cR^{z(T)}(T),\]
 where $z(T)$ is the smallest $n > 0$ such that $\cR^{n - 1}(T)$ and $\cR^{n}(T)$ have different type. 

Similarly, using $z: \Dom \to \N$ as the \emph{accelerating map}, we define the \emph{accelerated lengths and heights cocycles} 
 \[B^{-1}: \Dom \to SL(d, \Z), \hskip0.5cm B^{T}: \Dom \to SL(d, \Z),\]
by setting
 \[ B^{-1}(T) = A^{z(T)}(T)^{-1}, \hskip0.5cm B^{T}(T) = A^{z(T)}(T)^{T}. \]
As before, these cocycles are related to the transformation of lengths and heights under the action of $\cZ.$ Moreover, these cocycles will be integrable with respect to the invariant probability measure $\mu_{\cZ}$.

\subsubsection{Notations for iterates}
\label{sc: notations}

In the following, given $T= (\lambda, \pi) \in \Dom$, we denote its orbit under $\cZ$ by
\[ (\lambda^n, \pi^n) = \cZ^n(T).\]
We denote the type of $\cZ^n(T)$ by $\epsilon^n$ and the letters associated with its winner and loser intervals by $\alpha^n$ and $\beta^n$, respectively. We denote by $I^n$ the subinterval $I^n \subset I$, given by the Rauzy-Veech algorithm, such that $\cZ^n(T)$ coincides with the linear rescaling of $T$ when induced to $I^n$. We denote its associated decomposition by $\{ I_\alpha^n\}_{\alpha \in \A}.$ We denote by $h^n_\alpha$ the return time of $T$ to $I^n$ for any $x \in I^n_\alpha$. We define
\[ B^n(T) = B(T) \dots B({\cZ^n(T)}),\]
for any $n\geq 0$. 
Then, by definition of the accelerated lengths and heights cocycles, we have
\[ \lambda^n = B^n(T)^{-1}\lambda^0, \hskip0.5cm h^n = B^n(T)^Th^0,\]
where $h^0 = (1, \dots, 1) \in \N^\A$.

\subsubsection{Oseledet's splitting}
\label{sc: oseledets}

By Oseledet's theorem and the combination of several classical works  \cite{veech_gauss_1982, zorich_finite_1996, forni_deviation_2002, avila_simplicity_2007}, for $\mu_Z$-a.e. $(\lambda, \pi)$, there exist decompositions
\[ \R^{|\A|} = E_{d} \supsetneq \dots \supsetneq E_{d - g} = \dots = E_{g + 1} \supsetneq \dots \supsetneq E_0 = \{0\},\]
\[ \R^{|\A|} = F_{d} \supsetneq \dots \supsetneq F_{d - g} = \dots = F_{g + 1} \supsetneq \dots \supsetneq F_0 = \{0\},\]
invariant by the cocycles $(\cZ, B^T)$ and$(\cZ, B^{-1})$ respectively, corresponding to Lyapunov exponents 
\[\theta_g > \dots > \theta_1 > \theta_0 = 0 > -\theta_1 > \dots > -\theta_g.\]
We denote 
\[E^u = \R^{|\A|} = F^u, \hskip0.5cm E^{cs} = E_{2d - g}, \hskip0.5cm E^s = E_g, \hskip0.5cm F^{cs} = F_{2d - g}, \hskip0.5cm F^s = F_g.\]
 We have
\[ E^u \supset E^{cs} \supset E^s, \hskip1cm F^u \supset F^{cs} \supset F^s.\]
For a.e. $(\lambda, \pi) \in X_d$, the following holds.
\begin{itemize}
\item $ \lambda \in F^s(\lambda, \pi).$
\item $d - 2g = \dim(\textup{Ker}(\Omega_\pi))$ and $\textup{Ker}(\Omega_\pi) \subset F^{cs}(\lambda, \pi) \setminus F^s(\lambda, \pi).$
\item $E^s = (F^{cs})^\bot,$ and $F^s = (E^{cs})^\bot.$
\end{itemize}
 If $\pi \in \mathfrak{G}^0_d$ is of rotation type, then $\dim(\textup{Ker}(\Omega_\pi)) = d - 2.$ Hence, it follows form the previous relations that
 \[ \dim(E^{u}) = \dim(E^{s}) = 1 = \dim(F^{s}) = \dim(F^{u}),\]
 \[ F^s(\lambda, \pi) = \langle \lambda \rangle, \quad\quad E^{cs} = \orth{\lambda}, \quad\quad E^s = \textup{Ker}(\Omega_\pi)^\bot \cap \orth{\lambda}.\]

\subsection{P-homeomorphisms as GIETs}
\label{sc: P-homeomorphisms}

Recall that a homeomorphism of the circle $f: \T \to \T$ is called a \emph{P-homeomorphism} if it is a smooth orientation preserving homeomorphism, differentiable away from countable many points, so-called \emph{break-points}, at which left and right derivatives, denoted by $Df_-$, $Df_+$ respectively, exist but do not coincide, and such that $Df$ (which is defined away from break points) coincides with a function uniformly bounded from below and of bounded variation. A P-homeomorphism that is linear in each domain of differentiability is called \emph{PL-homeomorphism}. We denote the set of break points of a P-homeomorphism $f$ by 
\[BP(f) = \{x \in \T \mid Df_-(x) \neq Df_+(x) \}.\]
Since we often require P-homeomorphisms to have additional properties, we introduce the following notation. Define
\[ \Function{\varphi}{[0, 1)}{\T}{x}{e^{2\pi i x}}.\]
For any $d \geq 1$ and any $r \in [0, +\infty)$, let $P_d^r(\T)$ (resp. $PL_d(\T)$) be the space P-homeomorphisms (resp. PL-homeomorphisms) $f: \T \to \T$ such that:
	
\begin{enumerate}
\item $f$ is piecewise $C^r$ (resp. linear),
\item $\varphi(0) \in BP(f)$,
\item $|BP(f)| = d$,
\item $\rho(f) \in \R \setminus \Q$,
\item $f^n(x) \neq f^m(y)$ for any $n, m \in \Z$ and any $x, y \in BP(f)$, $x \neq y$.
\end{enumerate}
We treat P-homeomorphisms (resp. PL-homeomorphisms) as GIETs (resp. AIETs) using the circle parametrization given by the map $\varphi$. For any $f \in P^r_d(\T)$, the map
\[T_f = \varphi^{-1} \circ f \circ \varphi\]
is a well-defined GIET (resp. AIET) on $d + 1$ intervals. Since $f$ has exactly $d$ break-points lying in different orbits, $T_f$ defines an irrational GIET (resp. AIET). Moreover, $T_f$ cannot be seen as a GIET (resp. AIET) on a smaller number of intervals. In the following, for any $f \in P^r_d(\T)$ we define its \emph{combinatorial rotation number} as $\gamma(f) = \gamma(T_f)$. 

By Denjoy's theorem, a P-homeomorphism with irrational rotation number is topologically conjugated to a rigid rotation. In particular, given $f \in P^r_d(\T)$, the associated GIET $T_f$ has no wandering intervals. Hence, if $\gamma(f) = (\lambda, \pi) \in \Delta_d \times \mathfrak{G}^0_d$, then $T_f$ is topologically conjugated to the unique IET $T$ associated to $(\lambda, \pi).$

\begin{proposition}
Let $d \geq 1$ and $r \in [0, + \infty)$. Let $f \in P^r_d(\T)$ with rotation number $\alpha = \rho(f) \in \R \setminus \Q$ and combinatorial rotation number $\gamma(f) \in \Delta_d \times \mathfrak{G}^0_d.$ Then 
\begin{enumerate}
\item $f$ is topologically conjugated to $R_\alpha$ (as circle maps).
\item $T_f$ is topologically conjugated to $T= (\lambda, \pi)$ (as GIETs).
\end{enumerate}

\end{proposition}

\section{Statement of the main result}
\label{sc: results}

In this work, we aim to complement the results of K. Khanin and S. Kocić \cite{khanin_hausdorff_2017}, which concerns P-homeomorphisms with a finite number of breaks and nonzero mean nonlinearity, by considering the zero mean nonlinearity case. We will show that generically, the unique invariant probability measure of P-homeomorphisms with a finite number of breaks, irrational rotation number and zero mean nonlinearity has zero Hausdorff dimension. To encode this generic condition, we consider P-homeomorphisms as generalized interval exchange transformations (GIETs) of the interval and rely on the notion of \emph{combinatorial rotation number}, which can be seen as an extension of the classical notion of rotation number for circle homeomorphisms to the GIET setting.

Our main result is the following.

\begin{theorem}
\label{thm: HD_multibreak}
Let $d \geq 2$. There exists a full-measure set of combinatorial rotation numbers $\mathcal{C}_d \subset \Delta_{d + 1} \times \mathfrak{G}_{d + 1}^0$ such that, 
for any $f \in P_d^3(\T)$ with zero mean nonlinearity and $\gamma(f) \in \mathcal{C}_d$, the unique invariant probability measure $\mu_f$ of $f$ verifies $\textup{dim}_H(\mu_f) = 0.$
\end{theorem}

Let us point out that the nonzero mean nonlinearity hypothesis plays an essential role in the argument of \cite{khanin_hausdorff_2017} (which proves a similar result in the nonzero mean nonlinearity case) as the proof relies heavily on the behavior of renormalizations of P-homeomorphisms with a finite number of break points and nonzero mean nonlinearity. In fact, for a given map $f$ in this class, its renormalizations converge, in the $C^2$ norm, to a class of Möbius transformations whose second derivative is negative and \emph{uniformly} bounded away from zero. Exploiting this convergence, the authors show that the union of adjacent intervals in the RHS of (\ref{eq: refining_property}) accumulates `geometrically' near the boundary of $I^i_{n - 1} \setminus I^i_{n + 1}$, that is, their lengths decrease geometrically with respect to the length of $I^i_{n - 1}$.

Since all of the intervals in the RHS of (\ref{eq: refining_property}) have the same measure with respect to the unique invariant measure $\mu_f$ of $f$, the observation above allows to construct sets (more precisely Rohlin towers) with small Lebesgue measure (in fact small Hausdorff content) but whose measure with respect to $\mu_f$ tend to $1$. Refining this argument, the authors in \cite{khanin_hausdorff_2017} show that if an appropriate full-measure condition in the rotation number is satisfied, then the Hausdorff dimension of the unique invariant measure $\mu_f$ is zero. 

However, for circle diffeomorphisms with breaks and \emph{zero mean nonlinearity}, the renormalizations exhibit very different behavior. In fact, for simple examples in this class, such as piecewise affine circle homeomorphisms, the second derivative of any of their renormalizations is equal to $0$ everywhere. Furthermore, it follows from a recent result by S. Ghazouani and C. Ulcigrai \cite{ghazouani_priori_2021} that the renormalizations of any circle diffeomorphism with breaks of class $C^{2 + \epsilon}$ and zero mean nonlinearity converges, in $C^2$ norm, to the space of piecewise affine circle homeomorphisms. In particular, the second derivative of their renormalizations converges to $0$, and thus the argument in \cite{khanin_hausdorff_2017} cannot be extended to the zero mean nonlinearity case. 

As mentioned in the introduction, it follows from a recent work by P. Berk and the author \cite{berk_rigidity_2022}, zero mean nonlinearity P-homeomorphism of class $C^3$ without periodic points is $C^1$ conjugated to a PL- homeomorphism. Hence, Theorem \ref{thm: HD_multibreak} is a direct consequence of the following.

\begin{theorem}
\label{thm: HD_AIET}
Let $d \geq 2$. For a.e. $(\lambda, \pi) \in \Delta_d \times \mathfrak{G}_d^0$ of rotation type, and for any AIET $f$ on $d$ intervals with log-slope $\omega \in \R^d$ and $\gamma(f) = (\lambda, \pi)$, we have the following dichotomy, either
\begin{enumerate}
\item $f$ is $C^\infty$ conjugate to a standard IET, or
\item $\textup{dim}_H(\mu) = 0,$ where $\mu$ denotes the unique invariant probability measure $\mu_f$ for $f$.
\end{enumerate}
Moreover, the first condition is only verified if and only if $f$ can be seen as an AIET on two intervals, which corresponds to $\omega \in E^s(\lambda, \pi)$.
\end{theorem}





Let us point out that the aforementioned smooth linearization results in \cite{berk_rigidity_2022} rely heavily on the results of  \cite{ghazouani_priori_2021} concerning the behavior of subsequent renormalizations of circle maps without periodic points and zero mean nonlinearity. 
Let us mention that the behaviour on the renormalization of piecewise smooth circle maps was previously known only for transformations with \emph{bounded combinatorics} \cite{cunha_renormalization_2013}, \cite{dzhalilov_renormalizations_2021}. We refer the interested reader to \cite{cunha_renormalization_2013} or \cite{dzhalilov_renormalizations_2021} for the definition of bounded combinatorics.

\subsection{Strategy of proof}

Extracting the main elements of the strategy in \cite{khanin_hausdorff_2017} described in the previous section and refining the argument therein yields to the following criterion: For a ergodic piecewise continuous orientation preserving transformation $(T, \mu)$ on an interval, the existence of a sequence of `sufficiently rigid' Rohlin towers $\mathcal{F}_n $ (see conditions (\ref{eq: growth}) and (\ref{eq: suff_rigid}) in Proposition \ref{prop: criterion}) with intervals $F_n$ as bases, increasing heights $h_n$, measure $\mu(\mathcal{F}_n)$ uniformly bounded from below, and such that $T^{h_n}\mid_{F_n}$ is continuous and either contracts or expands at a uniform rate for all Rohlin towers, implies $\dim_H(\mu) = 0.$

More precisely, we have the following.

\begin{proposition}
\label{prop: criterion}
Let $T: [0, 1) \to [0, 1)$ be a piecewise smooth orientation preserving bijection and let $\mu$ be a $T$-invariant ergodic probability measure. Suppose there exist a sequence of intervals $F_n \subset [0, 1)$ and an increasing sequence of natural numbers $h_n$ 
such that 
\begin{enumerate}
 \item \label{cond: tower} $\mathcal{F}_n = \bigsqcup_{k = 0}^{h_n - 1} T^k(F_n)$ is a Rohlin tower, for any $n \geq 0$.
 \item \label{cond: smoothness} $T^{h_n}\mid_{F_n}$ is smooth, for any $n \geq 0$.
 \item \label{cond: measure} $\inf_{n \geq 0} \mu(\mathcal{F}_n) > 0.$
 \item \label{cond: uniform_slope} $\inf_{\substack{x \in F_n \\ n \geq 0 }} \big| DT^{h_n}(x) - 1\big| > 0.$ 
 \item \label{cond: rigidity} There exists a sequence of natural numbers $M_n$ obeying 
 \begin{equation}
 \label{eq: growth}
 \frac{M_n}{\log h_n} \to \infty,
 \end{equation}
 such that 
 \begin{equation}
 \label{eq: suff_rigid}
 \bigcap_{k = 0}^{M_n}T^{kh_n}(F_n) \neq \emptyset.
 \end{equation}
\end{enumerate}
Then, $\dim_H(\mu) = 0.$
\end{proposition}

Proposition \ref{prop: criterion} will be proven in Section \ref{sc: proof_criterion}. The proof of Theorem \ref{thm: HD_multibreak} (or more precisely, Theorem \ref{thm: HD_AIET} of which Theorem \ref{thm: HD_multibreak} will be a consequence) will be an application of the previous criterion.

To prove Theorem \ref{thm: HD_multibreak}, it is enough to consider the case of PL-homeomorphisms. In fact, it follows from a recent result by P. Berk and the author that, generically, a P-homeomorphism of class $C^{2 + \epsilon}$ with zero mean nonlinearity is $C^1$ conjugated to a PL-homeomorphism. Let us point out that this linearization result was previously shown by K. Cunha and D. Smania \cite{cunha_rigidity_2014} in the particular case of P-homeomorphisms with bounded combinatorics. 

To show that generic PL-homeomorphisms fulfill the hypotheses of Proposition \ref{prop: criterion}, we will use renormalization techniques for interval exchange transformations (IETs). We treat PL-homeomorphisms as affine interval exchange transformations (AIETs) by parametrizing the circle $\T$ as
\[ \Function{\varphi}{[0, 1)}{\T}{x}{e^{2\pi i x}},\]
and restricting ourselves to PL-homeomorphisms such that $\varphi(0)$ is a break-point of $f$. Then, if $f$ has $d \geq 1$ break points, the map $\varphi^{-1} \circ f \circ \varphi$ can be seen as a well-defined AIET on $d + 1$ intervals. See Section \ref{sc: P-homeomorphisms} for more details on this identification. 

Theorem \ref{thm: HD_multibreak} will be a direct consequence of Theorem \ref{thm: HD_AIET}, which is an analogous result in the case of AIETs of rotation type. For clarity, we postpone a precise statement to Section \ref{sc: results}.

\section{Proof of the zero HD criterion}
\label{sc: proof_criterion}

The following lemma is a well-known fact. 

\begin{lemma}
\label{lemma: full_measure_rokhlin}
Let $(T, X, \mu)$ be an ergodic measure preserving transformation on a probability space. Then, for any $c > 0$ and any sequence of Rokhlin towers 
\[\twr_k := \twr(F_k, h_k) = \bigcup_{i = 0}^{h_k - 1} T^i(F_k),\]
 with $h_k \to +\infty$ and $\mu(\twr_k) > c,$
\[ \mu_T \left(\bigcap_{n \geq 0} \bigcup_{k \geq n} \twr_{k} \right) = 1.\]
\end{lemma}

\begin{proof}
Let \[A = \bigcap_{n \geq 0} \bigcup_{k \geq n} \twr_{k}.\]
Since $A$ is the intersection of a decreasing sequence of sets of measure at least $c$, it follows that $\mu(A) \geq c.$ Since $\mu$ is an ergodic $T$-invariant measure, it suffices to show that $A$ is a $T$-invariant set. Notice that
\[ A \Delta T^{-1}(A) \subset \bigcup_{k \geq n} T^{-1}(F_k) \cup T^{h_k - 1}(F_k),\]
for any $n \in \N.$ Up to take a subsequence and since $h_n \to +\infty$, we may assume
\[ \sum_{n \geq 0} \mu(F_n) < +\infty.\]
Therefore 
\[ \mu(A \Delta T^{-1}(A)) \leq \lim_{n \to \infty} 2 \sum_{k \geq n} \mu(F_k) = 0. \]
Hence $A$ is $T$-invariant set. By ergodicity, $\mu(A) = 1$. 
\end{proof}

We are now in a position to prove Proposition \ref{prop: criterion}. 

\begin{proof}[Proof of Proposition \ref{prop: criterion}]
For any $n \geq 0$, let us denote the left and right endpoints of $F_n$ by $l_n$ and $r_n$, respectively. Notice that (\ref{eq: suff_rigid}), together with the continuity of $T^{h_n}\mid_{F_n}$, imply that either 
\begin{equation}
\label{eq: nb_iterates_lower}
 T^{kh_n}(l_n) \in F_n, \quad \text{ for all }0 < k \leq M_n,
\end{equation}
 or
 \[ T^{kh_n}(r_n) \in F_n, \quad \text{ for all }0 < k \leq M_n.\]
Clearly, one of the two equations above must hold for infinitely many values of $n$. Hence, up to considering a subsequence, we may assume WLOG that one of the two equations holds for all $n \geq 0$. From now on, and for the sake of simplicity, let us assume that the first of the two equations holds for all $n \geq 0$, the other case being analogous. Moreover, by taking $M_n$ bigger if necessary, we may assume that 
\begin{equation}
\label{eq: nb_iterates_upper}
T^{(M_{n} + 1)h_n}(l_n) \notin F_n.
\end{equation}
Similarly, by (\ref{cond: uniform_slope}), we may assume WLOG that either $DT^{h_n}\mid_{F_n}$ is uniformly bigger than one for all $n \geq 0$ or it is uniformly smaller than one for all $n\geq 0$. For the sake of simplicity, let us assume that
\[ \sigma = \sup_{\substack{x \in F_n \\ n \geq 0 }} DT^{h_n}(x) < 1,\]
the other case being analogous. 

Let $(L_n)_{n \geq 0}$ be a sequence of natural numbers such that 
\[ \frac{L_n}{\log h_n} \to \infty, \hskip1cm \frac{L_n}{M_n} \to 0,\]
and define
\[ G_n =\bigsqcup_{i = L_n}^{M_n - 1} T^{ih_n} \left( \big(l_n, T^{h_n}(l_n)\big)\right), \hskip1cm \mathcal{G}_n = \bigsqcup_{j = 0}^{h_n - 1} T^{j}(G_n)\]
 \[ \mathcal{X}_n = \bigcup_{k \geq n} \mathcal{G}_k,\hskip1cm \mathcal{X} = \bigcap_{n \geq 0} \mathcal{X}_n,\]
for any $n \geq 0$. We will show that $\mu(\mathcal{X}) = 1$ and $\dim_H(\mathcal{X}) = 0$.

Notice that $G_n \subset F_n$ and $\mathcal{G}_n \subset \mathcal{F}_n$. We shall see that, although $\mathcal{G}_n$ has a very small Hausdorff content, its $\mu$-measure is comparable to that of $\mathcal{F}_n$. More precisely, we can show the following.
\begin{claim}
For any $0 < s < 1$, there exists $C > 0$ such that \[C^s_H(\mathcal{G}_n) \leq C e^{\frac{\log \sigma }{2} L_n}\]
for any $n \geq 0$. Moreover, 
$$\inf_{n \geq 0} \mu(\mathcal{G}_n) > 0.$$
\end{claim}

Before proving this claim, let us show how to conclude the proof of the proposition. By Lemma \ref{lemma: full_measure_rokhlin} and the previous claim, $\mu(\mathcal{X}) = 1.$ Moreover, up to take a subsequence, we may assume WLOG that 
\[ \sum_{k \geq 0} C_H^s(\mathcal{G}_k) < +\infty. \]
Thus
\[ C_H^s (\mathcal{X}) \leq \liminf_{n\to \infty } C_H^s (\mathcal{X}_n) \leq \liminf_{n\to \infty }\sum_{k \geq n} C_H^s(\mathcal{G}_k) = 0,\]
for any $0 < s < 1$. Hence,
 \[ \dim_H(\mathcal{X}) = \inf\{s > 0 \mid C_H^s(\mathcal{X}) = 0 \} = 0.\]
Therefore,
\[ \dim_H(\mu) = \inf \{ \dim_H(X) \mid X \subset [0, 1);\, \mu(X) = 1\} = 0.\]

\begin{proof}[Proof of the Claim] Fix $n \geq 0$. Then
\[|G_n| = \sum_{i = L_n}^{M_n} \left| T^{ih_n} \big( \big(l_n, T^{h_n}(l_n)\big)\big) \right| \leq \sum_{i = L_n}^{M_n} \sigma^i \big| \big(l_n, T^{h_n}(l_n)\big)\big| \leq C_\sigma |F_n|\sigma^{L_n}, \]
where $C_\sigma = \frac{1 - \sigma^{M_n - L_n}}{1 - \sigma}.$ Moreover, it follows from simple bounded distortion arguments that
\[ |T^j(G_n)| \leq C_\sigma C_T \sigma^{L_n} |T^j(F_n)|,\]
for any $0 \leq j <h_n$, where $C_T = \max_{x \in [0, 1)} \frac{T''(x)}{T'(x)}$. 

Hence, for any $0 < s < 1$, 
\begin{align*}
C_H^s(\mathcal{G}_n) & \leq \sum_{j = 0}^{h_n - 1} |T^j(G_n)|^s \\
& \leq C_\sigma C_T \sigma^{L_n} \sum_{j = 0}^{h_n - 1} |T^j(F_n)|^s \\
& \leq C_\sigma C_T \sigma^{L_n} h_n^{1 - s} \\
& = C_\sigma C_T \exp(L_n \log \sigma + (1 - s) \log h_n) \\
& \leq C_\sigma C_T C_{h, s} \exp\big(\tfrac{\log \sigma}{2} L_n \big),
\end{align*}
for some positive constant $C_{h,s}$ independent of $n$. 

By (\ref{eq: nb_iterates_lower}) and (\ref{eq: nb_iterates_upper}), 
\[ M_n\mu((l_n, T^{h_n}(l_n))) \leq \mu(F_n) \leq (M_n + 1)\mu((l_n, T^{h_n}(l_n))).\]
Hence
\[ \frac{M_n - L_n}{M_n + 1} \leq \frac{\mu(G_n)}{\mu(F_n)} \leq \frac{M_n - L_n}{M_n}, \]
which implies 
\[ \frac{\mu(\mathcal{G}_n)}{\mu(\mathcal{F}_n)} \to 1.\]
\end{proof}
This finishes the proof of the proposition.
\end{proof}

\section{Proof of Theorem \ref{thm: HD_AIET}}

For the sake of simplicity, let us start by introducing some notations that will be used in the remaining of this work.

For any $d \geq 2$, we denote by $\pi^* = (\pi^*_0, \pi^*_1) \in \mathfrak{G}_d^0$ a fixed combinatorial datum verifying
\begin{equation}
\label{eq: rotation}
 \pi_1^* \circ {\pi_0^*}^{-1}(1) = d, \hskip1cm \pi_1^* \circ {\pi_0^*}^{-1}(k) = k - 1,
\end{equation} 
for $k = 2, \dots, d$. Clearly, any combinatorial datum verifying the previous equation is of rotation type. Notice that, although a permutation $\pi^* \in \mathfrak{G}_d^0$ verifying (\ref{eq: rotation}) always exists, $\pi^*$ is not necessarily unique. In fact, for $d = 4$, the permutations
\[\begin{pmatrix} A & B & C & D \\ B & C & D & A \end{pmatrix}, \hskip1cm \begin{pmatrix} A & C & B & D \\ C & B & D & A \end{pmatrix},\]
verify (\ref{eq: rotation}). We will denote the last letters in the top and bottom rows of $\pi^*$ by $\alpha^* = {\pi_0^*}^{-1}(d)$ and $\beta^* = {\pi_1^*}^{-1}(d)$. 

 With these notations, we can explicitly state the generic condition in Theorem \ref{thm: HD_AIET}.

\begin{lemma}
\label{lemma: generic_condition}
For any $0 < c_0 < \frac{1}{10d}$ sufficiently small and for any increasing sequence $\{C(n)\}_{n \in \N} \subset \N$ verifying
\begin{equation}
\label{eq: conditions_increasing_sequence}
\sum_{n \geq 1} \frac{1}{n C(n)} = + \infty,
\end{equation}
the following holds.

 For a.e. $(\lambda, \pi) \in \Delta_d \times \mathfrak{G}_d^0$ of rotation type, there exists an increasing sequence $\{n_k\}_{k \in \N} \subset \N$ such that:
\begin{enumerate}
\item \label{item: combinatorics} $\pi^{(n_k)} = \pi^*$, with $\pi^* \in \mathfrak{G}_d^0$ as in (\ref{eq: rotation}),
\item \label{item: lengths} $\lambda_\alpha^{(n_k)} > n_kC(n_k)\lambda^{(n_k)}_{\beta^*}$, for all $\alpha \neq \beta^*$,
\item \label{item: lengths_balanced} $\lambda^{(n_k)}_{\alpha} > c_0,$ for all $\alpha \neq \beta^* $.
\item \label{item: heights} $\frac{h^{(n_k)}_\alpha}{h^{(n_k)}_\beta} > c_0$, for all $\alpha, \beta \in \A.$
\end{enumerate}
\end{lemma}
Before proving the lemma above, let us mention how these conditions will appear in the proof of (the second assertion of) Theorem \ref{thm: HD_AIET}. As mentioned before, this assertion will be a consequence of Proposition \ref{prop: criterion}, for which an appropriate sequence of Rohlin towers is required. Thus, in order to build towers for a given AIET $f$ with $\gamma(f) = (\lambda, \pi)$, we start by building towers for the IET $T$ defined by $(\lambda, \pi)$, to which $f$ is conjugated. As we shall see, these towers will already verify some of the hypotheses of Proposition \ref{prop: criterion}.

 Given $T= (\lambda, \pi)$ as in Lemma \ref{lemma: generic_condition}, and using the notations in Section \ref{sc: notations}, it follows from the renormalization procedure that, for any $\alpha \in \A$, the set $\bigsqcup_{j = 0}^{q_\alpha^{(n_k)}} T^j\big( I_\alpha^{(n_k)}(T)\big)$ can be seen as a well defined Rohlin tower (see Condition \ref{cond: tower}) such that $T^{q_\alpha^{(n_k)}}\mid_{ I_\alpha^{(n_k)}(T)}$ is smooth (see Condition \ref{cond: smoothness}). The first two assertions in the previous lemma guarantee that, for $\alpha \neq \beta^*$, a big number of iterates (approx. $n_kC(n_k)$) of $I^{(n_k)}_{\beta^*}$ with respect to $T^{(n_k)}$ are contained in each of the intervals $I^{(n_k)}_\alpha$ and, on each of them, consecutive iterates appear as adjacent intervals (see Condition \ref{cond: rigidity}). The last three assertions in the lemma guarantee that the Lebesgue measure of the Rohlin tower associated with each of the intervals $I^{(n_k)}_\alpha,$ for $\alpha \neq \beta^*$, is bounded from below by a uniform constant (see Condition \ref{cond: measure}).

More precisely, from the previous lemma, we immediately conclude the following. 

\begin{corollary}
\label{cor: generic_condition}
Let $T = (\lambda, \pi) \in \Delta_d \times \mathfrak{G}_d^0$ as in Lemma \ref{lemma: generic_condition}. There exists an increasing sequence $\{n_k\}_{k \in \N} \subset \N$ such that the following holds.
 \begin{enumerate}
\item For any $k \in \N$, there exist natural numbers $1 = l_0^k < l_1^k < \dots < l_{d - 1}^k$ such that $l_{j + 1}^k - l_j^k \geq {n_kC(n_k) - 2}$ and 
\[ \bigsqcup_{i = l_j^k + 1}^{l_{j + 1}^k - 1} {T^{(n_k)}}^i \left( I^{(n_k)}_{\beta^*} \right) \subset I^{(n_k)}_{{\pi_0^*}^{-1}(d - j)},\]
for $j = 1, \dots, d- 1.$ In particular 
\[ \bigcap_{i = 0}^{n_kC(n_k) - 2} {T^{(n_k)}}^i \left( I^{(n_k)}_{{\pi_0^*}^{-1}(d - j)} \right) \neq \emptyset,\]
or $j = 1, \dots, d- 1.$
\item There exists a constant $c > 0$, not depending on $k$ or $\alpha$, such that 
\[ \min_{\alpha \neq \beta^*} \sum_{i = 0}^{q^{(n_k)}_\alpha} \left| T^i \left( I_\alpha^{(n_k)} \right)\right| = \min_{\alpha \neq \beta^*} q_\alpha^{(n_k)} | I_\alpha^{(n_k)}| > c,\]
for all $k \in \N$.
\end{enumerate}
\end{corollary}

Therefore, if $f$ is an AIET with $\gamma(f) = (\lambda, \pi)$ and verifying Lemma \ref{lemma: generic_condition}, it follows from the Corollary \ref{cor: generic_condition} that, for any $\alpha \neq \beta^*$, the sequence of Rohlin towers 
\begin{equation}
\label{eq: AIET_Rohlin_towers}
\bigsqcup_{j = 0}^{q_\alpha^{(n_k)}} f^j\big( I_\alpha^{(n_k)}(f)\big)
\end{equation}
verify all of the hypotheses of Proposition \ref{prop: criterion} except for Condition \ref{cond: uniform_slope}. 

In fact, since $f$ and $T$ are conjugated by some homeomorphism $h \in \textup{Hom}([0, 1)),$ verifying $f \circ h = h \circ T$, then 
\[h \left( \bigsqcup_{j = 0}^{q_\alpha^{(n_k)}} T^j\big( I_\alpha^{(n_k)}(T)\big) \right) = \bigsqcup_{j = 0}^{q_\alpha^{(n_k)}} f^j\big( I_\alpha^{(n_k)}(f)\big),\]
and
\[ \mu_f\left( \bigsqcup_{j = 0}^{q_\alpha^{(n_k)}} f^j\big( I_\alpha^{(n_k)}(f)\big) \right) = \sum_{i = 0}^{q^{(n_k)}_\alpha} \left| T^i \left( I_\alpha^{(n_k)}(T) \right)\right|,\]
where $\mu_f$ denotes the unique invariant probability measure of $f$. 

Therefore, since the towers $\bigsqcup_{j = 0}^{q_\alpha^{(n_k)}} T^j\big( I_\alpha^{(n_k)}(T)\big)$ verify Conditions \ref{cond: tower}, \ref{cond: smoothness}, \ref{cond: measure} and \ref{cond: rigidity} in Proposition \ref{prop: criterion}, it is readily seen that the towers in (\ref{eq: AIET_Rohlin_towers}) also verify these conditions. Notice that, for a fixed $\alpha \neq \beta^*$, Condition \ref{cond: uniform_slope} for the towers in (\ref{eq: AIET_Rohlin_towers}) is equivalent to
\begin{equation}
\label{eq: log_slope_cond}
\inf_{k \geq 1} \big| \omega_\alpha^{(n_k)}\big| > 0,
\end{equation}
where $ \omega_\alpha^{(n_k)}$ is the log-slope vector of $f^{(n_k)}$. Recall that if $f$ has log-slope vector $\omega$, then $ \omega^{(n_k)} = B^n(\lambda, \pi)^{T}\omega$ by Proposition \ref{prop: log_slope}. 

As we shall see (Corollary \ref{cor: central_bound_below}), if the log-slope vector of $f$ does not belong to the stable space $E^s(\lambda, \pi)$ then, up to considering a subsequence, (\ref{eq: log_slope_cond}) is satisfied for at least one $\alpha \neq \beta^*$, and thus the second assertion of Theorem \ref{thm: HD_AIET} would follow by Proposition \ref{prop: criterion} when applied to the towers given by (\ref{eq: AIET_Rohlin_towers}) for this particular $\alpha$.

To see that such $\alpha \in \A \setminus \{\beta^*\}$ indeed exists, we will use the following properties of the lengths cocycle (see Section \ref{sc: lengths_cocycle} for the definition). 

\begin{lemma}
\label{lemma: kernel_projection}
Let $(\lambda, \pi) \in \Delta_d \times \mathfrak{G}^0_d$ infinitely renormalizable with $\pi$ of rotation type. Then, for any $\omega \in \R^\A$ and for any $n \in \N$ such that $\pi^{(n)} = \pi$,
\[ \pi_{\textup{Ker}(\Omega_\pi)}\big(\omega^{(n)}\big) = \pi_{\textup{Ker}(\Omega_\pi)}(\omega),\]
where $\pi_{\textup{Ker}(\Omega_\pi)}: \R^\A \to \textup{Ker}(\Omega_\pi)$ is the orthogonal projection to $\textup{Ker}(\Omega_\pi)$ and $\omega^{(n)} = A^n(\lambda, \pi)^{-1}\omega.$ Moreover, $$\pi_{\textup{Ker}(\Omega_\pi)}(\omega) \neq 0,$$ for any $\omega \in E^{cs}(\lambda, \pi) \setminus E^{s}(\lambda, \pi)$.

\end{lemma} 

\begin{proof}
Fix $(\lambda, \pi) \in \Delta_d \times \mathfrak{G}^0_d$ infinitely renormalizable with $\pi$ of rotation type and let $\omega \in \R^\A.$ Then, for any $n \in \N$ such that $\pi^{(n)} = \pi$, and for any $v \in \textup{Ker}(\Omega_\pi)$,
\begin{align*}
\left| \langle \omega, v \rangle \right| & = \left| \left\langle \left({B^{(n)}}^T\right)^{-1} \omega^{(n)}, v \right
\rangle \right| \\
& = \big| \big\langle \omega^{(n)}, \big(B^{(n)}\big)^{-1}v \big\rangle \big|\\
& = | \langle \omega^{(n)}, v \rangle|,
\end{align*}
where the last equality follows from Proposition \ref{prop: trivial_action}. Hence, $$\pi_{\textup{Ker}(\Omega_\pi)}\big(\omega^{(n)}\big) = \pi_{\textup{Ker}(\Omega_\pi)}(\omega),$$ for all $n \in \N$. 

Recall that (see Section \ref{sc: oseledets}), for $\pi$ of rotation type, we have
\[
\begin{array}{ll}
 E^{cs}(\lambda, \pi) = \orth{\lambda}, & \textup{dim}\big(E^{cs}(\lambda, \pi)\big) = d - 1, \\ E^{s}(\lambda, \pi) = \textup{Ker}(\Omega_\pi)^\bot \cap \orth{\lambda}, & \textup{dim}\big(E^{s}(\lambda, \pi)\big) = 1.
 \end{array}
 \]
Hence
\[\big(E^{cs}(\lambda, \pi) \setminus E^{s}(\lambda, \pi) \big) \cap \textup{Ker}(\Omega_\pi)^\bot = \{0\},\] 
since otherwise $1 = \textup{dim}\big( \textup{Ker}(\Omega_\pi)^\bot \cap \orth{\lambda} \big) > 1$. Thus, $$\pi_{\textup{Ker}(\Omega_\pi)}(\omega) \neq 0,$$
 for any $\omega \in E^{cs}(\lambda, \pi) \setminus E^{s}(\lambda, \pi)$.
\end{proof}

As a simple consequence of Lemmas \ref{lemma: generic_condition} and \ref{lemma: kernel_projection}, we have the following. 

\begin{corollary}
\label{cor: central_bound_below}
Let $(\lambda, \pi) \in \Delta_d \times \mathfrak{G}^0_d$ as in Lemma \ref{lemma: generic_condition}. Then, for any $\omega \in E^{cs}(\lambda, \pi) \setminus E^{s}(\lambda, \pi)$, there exists a constant $c_2 > 0$, depending only on $\omega$, such that
\[ \inf_{k \geq 1} \max_{\alpha \neq \beta^*} \big| \omega^{(n_k)}_\alpha \big| > c_2, \]
where $\{n_k\}_{k \in \N} \subset \N$ is the sequence in Lemma \ref{lemma: generic_condition}.
\end{corollary}

\begin{proof}
Fix $k \geq 1$. By definition of $\pi^*$, we have
\begin{equation}
\label{eq: kernel}
\textup{Ker}(\Omega_{\pi^*}) = \left\{ v \in \R^\A \,\left|\, v_{\beta^*} = 0;\, v_{\alpha^*} = -\sum_{\delta \neq \alpha^*} v_\delta \right. \right\}.
\end{equation}
Then 
\[\pi_{\textup{Ker}(\Omega_{\pi^*})}\big( \pi_{\orth{e_{\beta^*}}}\big(\omega^{(n_k)}\big) \big) = \pi_{\textup{Ker}(\Omega_{\pi^*})}\big(\omega^{(n_k)}\big),\]
where $\orth{e_{\beta^*}} = \big\{ v \in \R^\A \,\big|\, v_{\beta^*} = 0\big\}$ and $\pi_{\orth{e_{\beta^*}}}: \R^\A \to \orth{e_{\beta^*}}$ denotes the orthogonal projection to $\orth{e_{\beta^*}}$. By Lemma \ref{lemma: kernel_projection},
\[\pi_{\textup{Ker}(\Omega_{\pi^*})}\big( \pi_{\orth{v_{\beta^*}}}\big(\omega^{(n)}\big) \big) = \pi_{\textup{Ker}(\Omega_{\pi^*})}(\omega) \neq 0.\]
Hence, there exists $c > 0$, depending only on $\pi^*$ and $\omega$, such that 
\[ \left| \pi_{\orth{v_{\beta^*}}}\big(\omega^{(n)}\big) \right| > c.\]
\end{proof}

We are now in a position to prove Theorem \ref{thm: HD_AIET}. 

\begin{proof}[Proof of Theorem \ref{thm: HD_AIET}]
Let $d \geq 2$. It follows from the results in \cite{cobo_piece-wise_2002} that for a.e. IET $T = (\lambda, \pi) \in \Delta_d \times \mathfrak{G}_d^0$ (not necessarily of rotation type) and for any AIET $f$ with $\gamma(f) = (\lambda, \pi)$, its log-slope vector $\omega \in \R^\A$ verifies $\omega \in E^{cs}(\lambda, \pi)$. Moreover, if $\omega \in E^s(\lambda, \pi)$ then $f$ is $C^\infty$ conjugated to $T$.

Let $\mathfrak{R}$ denote the Rauzy class of rotation type permutations in $\mathfrak{G}_d^0$. By the ergodicity of the Zorich map on $\Delta_d \times \mathfrak{R}$ and the fact that the Hausdorff dimension of the unique invariant measure of a uniquely ergodic AIET $f$ and of its renormalization $\cZ(f)$ is the same (since $\cZ(f)$ is just $f$ induced on a subinterval and the returns to this subinterval are uniformly bounded), it is sufficient to prove the theorem for IETs of rotation type with combinatorial datum given by $\pi^*$, with $\pi^*$ as in (\ref{eq: rotation}) (in fact it is sufficient to prove the theorem for any fixed rotation type combinatorial datum). 

By definition of $\pi^*$, we have
\begin{equation}
\label{eq: kernel_orth}
\textup{Ker}(\Omega_{\pi^*})^\bot = \textup{Vect} \left( e_{\beta^*}, \sum_{\delta \neq \beta^*} e_\delta \right).
\end{equation}
Recall that for a.e. $(\lambda, \pi) \in \Delta_d \times \mathfrak{R}$ we have
\begin{equation}
\label{eq: cs_s}
E^{cs}(\lambda, \pi) = \orth{\lambda}, \hskip1cm E^s(\lambda, \pi) = \orth{\lambda} \cap
\textup{Ker}(\Omega_\pi)^\bot.
\end{equation}
It follows from (\ref{eq: rotation}), (\ref{eq: kernel_orth}) and (\ref{eq: cs_s}) that any AIET $f$ with $\gamma(f) = (\lambda, \pi^*)$ and log-slope vector $\omega \in \R^\A$ can be seen as a 2-IET if and only if $\omega \in E^s(\lambda, \pi^*)$. 

Taking into account these observations, the theorem will be proved if we show that, for a.e. $\lambda \in \Delta_d$ and for any AIET $f$ with $\gamma(f) = (\lambda, \pi^*)$ and log-slope $\omega \in E^{cs}(\lambda, \pi^*) \setminus E^s(\lambda, \pi^*)$, the Hausdorff dimension of its unique invariant measure $\mu_f$ is equal to zero.

Consider $(\lambda, \pi^*) \in \Delta_d \times \mathfrak{R}$ for which Lemma \ref{lemma: generic_condition} (and in particular Corollary \ref{cor: generic_condition}) and Corollary \ref{cor: central_bound_below} hold. By Corollary \ref{cor: central_bound_below} and up to taking a subsequence, we can assume WLOG that there exists $\alpha \in \A \setminus \{\beta^*\}$ such that (\ref{eq: log_slope_cond}) holds. Notice that the set of such $\lambda$ defines a full-measure set in $\Delta_d$.

Let $f$ be an AIET with combinatorial rotation number $\gamma(f) = (\lambda, \pi^*)$, log-slope $\omega \in E^{cs}(\lambda, \pi^*) \setminus E^s(\lambda, \pi^*)$ and unique invariant measure $\mu_f$. Then, $\dim_H(\mu_f) = 0$ by applying Proposition \ref{prop: criterion} to the towers given by (\ref{eq: AIET_Rohlin_towers}), which are defined using Corollary \ref{cor: generic_condition}. Notice that these towers verify the hypotheses in Proposition \ref{prop: criterion} by Corollaries \ref{cor: generic_condition} and \ref{cor: central_bound_below}.
\end{proof}

The remaining of this work concerns the proof Lemma \ref{lemma: generic_condition}.

\subsection{Proof of Lemma \ref{lemma: generic_condition}}

Lemma \ref{lemma: generic_condition} will be an application of the following ``Borel-Cantelli lemma" for the Zorich map, due to R. Aimino et al. \cite[Theorem 2.18]{aimino_recurrence_2017}.

Before stating (a simplified version of) this result, let us introduce some notation. Let $\mathfrak{R}$ be a fixed Rauzy class. Let 
\[ \mathcal{P}_\mathfrak{R} = \{ \Delta_{\pi, 0} \times \{\pi\}, \Delta_{\pi, 1} \times \{\pi\} \mid \pi \in \mathfrak{R}\},\]
and define
\[ \mathcal{P}_\mathfrak{R}^n = \bigvee_{i = 0}^{n - 1} \mathcal{R}^{-i}(\mathcal{P}_\mathfrak{R}),\]
for any $n \geq 1$. Notice that any $B \in \mathcal{P}_\mathfrak{R}^n$ is contained in a unique simplex of the form $$\Delta_\pi^\epsilon = \Delta_{\pi, \epsilon} \times \{\pi\},$$ for some $\pi \in \mathfrak{R}$ and $\epsilon \in \{ 0, 1\}$. For any $\epsilon \in \{0, 1\}$, we denote 
\[\Delta_\mathfrak{R}^\epsilon = \bigcup_{\pi \in \mathfrak{R}} \Delta_\pi^\epsilon.\]
Given $A \subset \Delta_\pi^\epsilon$, we denote by $\partial A$ its boundary with respect to the usual Euclidean distance in $\Delta_{\pi, \epsilon}$. Similarly, for any $\delta > 0$, we denote by $B_\delta(A)$ the $\delta$-neighbourhood of $A$ with respect to the usual Euclidean distance in $\Delta_{\pi, \epsilon}$.

\begin{theorem}[Borel-Cantelli lemma for Zorich acceleration {\cite{aimino_recurrence_2017}}] 
\label{thm: borel_cantelli}
Let $\mathfrak{R}$ be a fixed Rauzy class, $n \geq 1$ and $\epsilon \in \{0, 1\}$. Suppose $B \in \mathcal{P}_\mathfrak{R}^n$ verifies $B \subset \Delta_\mathfrak{R}^\epsilon$ and $\overline{B} \subset \Delta_d \times \mathfrak{R}$. 

Then, for any sequence $\{ A_n\}_{ n \geq 1}$ of subsets of $B$ such that 
\[\sum_{n \geq 1} \mu(A_n) = +\infty, \hskip1cm \sup_{n \geq 1}\, \overline{\lim_{\delta \searrow 0}} \frac{|B_\delta(\partial A_n)|}{\delta^\alpha} < +\infty,\]
for some $0 < \alpha < 1$, we have
\[ \frac{1}{E_n} \sum_{i = 1}^n \chi_{A_i} \circ \mathcal{Z}^{2i} (x) \to 1, \]
for $\mu$-a.e. $x \in \Delta_\mathfrak{R}^\epsilon$, where $E_n = \sum_{i = 1}^n \mu(A_i)$.
\end{theorem}


We are now in position to prove Lemma \ref{lemma: generic_condition}.

\begin{proof}[Proof Lemma \ref{lemma: generic_condition}] Let $\pi_* \in \mathfrak{G}_d^0$ be the pre-image of $\pi^*$ by a bottom movement in the Rauzy graph. Denote $\delta^*$ the last letter in the top row of $\pi_*$. Notice that the last letter in the bottom rows of $\pi^*$ and $\pi_*$ is $\beta^*$.

In the following, we denote by $\mathfrak{R}$ the set of rotation type permutations. Let $\gamma$ be a finite path in the Rauzy graph, starting with a top movement and ending at $\pi_*$ through a bottom type movement, such that the associated matrix $A_\gamma$ is positive. 
We denote by $\Delta_\gamma$ the simplex in $\Delta_d \times \mathfrak{R}$ defined by the $(\lambda, \pi) \in \Delta_d \times \mathfrak{R}$ following the path $\gamma$ under Rauzy-Veech induction. 

Let $0 < c_0 < \frac{1}{10d}$ and define
\[ \widehat{A}_n = \left\{ \lambda \in \Delta_d \,\left|\, \min\left\{c_0, \frac{1}{nC(n)} \right\}> \lambda_{\beta^*} - \lambda_{\delta^*} > 0; \, \min_{\alpha \in \A} \lambda_\alpha > c_0 \right. \right\},\]
for any $n \geq 1$. Clearly
\begin{equation}
\label{eq: lower_bound_set}
|\widehat{A}_n| \geq \frac{c_1}{nC(n)}
\end{equation}
for some $c_1 > 0$ depending only on $c_0$ and $d$. Notice that for any $\lambda \in \widehat{A}_n$, $(\lambda, \pi_*)$ is of bottom type. Moreover, its RV-renormalization $(\lambda', \pi^*)$ is of top type and verifies
\begin{equation}
\label{eq: lengths}
\min_{\alpha \neq \beta^*} \lambda_\alpha' > \max\left\{ c_0, c_d n C(n)\lambda_{\beta^*}' \right\},
\end{equation}
for some $c_d > 0$ depending only on $d$. We assume WLOG that $c_d = 1$. Otherwise, it suffices to slightly modify the definition of $\widehat{A}_n$ to compensate for the constant $c_d$.

\begin{claim}
For any $n \geq 1$, let
\[ A_n = \left\{\lambda \in \Delta_\gamma \, \left| \, \frac{A_\gamma^{-1} \lambda}{\big| A_\gamma^{-1} \lambda \big|_1} \in \widehat{A}_n \right. \right\}. \]
There exists $N \in \N$ and $c_\gamma > 0$, depending only on $\gamma$, such that for any $(\lambda, \pi) \in A_n$, $\pi^{(N)} = \pi^*$,
\begin{equation}
\label{eq: length_estimates}
\min_{\alpha \neq \beta^*} \lambda_\alpha^{(N)} > \max\left\{ c_0, n C(n)\lambda^{(N)}_{\beta^*} \right\}
\end{equation}
and
\begin{equation}
\label{eq: height_estimates}
\min_{\alpha, \beta \in \A}\frac{h^{(N)}_\alpha }{h^{(N)}_\beta} > c_\gamma.
\end{equation}

\end{claim}
\begin{proof}[Proof of the Claim] Fix $n \geq 1$. Notice that any $(\lambda, \pi) \in A_n$ follows the path $\gamma' = \gamma \ast \{ \pi^*\}$ under Rauzy-Veech induction. Furthermore, there exists $1 < N \leq |\gamma| + 1$, depending only on $\gamma$, such that $(\lambda^{(N)}, \pi^{(N)})$ coincides with the $|\gamma'|$-th RV renormalization of $(\lambda, \pi)$. 

Hence $\mathcal{R}^{|\gamma|}(\lambda, \pi) \in \widehat{A}_n \times \{ \pi_*\}$ and $\mathcal{R}^{|\gamma| + 1}(\lambda, \pi) = (\lambda^{(N)}, \pi^{(N)})$. Therefore, $\pi^{(N)} = \pi^*$ and (\ref{eq: length_estimates}) follows directly from (\ref{eq: lengths}). Since $A_{\gamma'}$ is positive, that is, if $A_{\gamma'} = (a_{\alpha \beta})_{\alpha, \beta \in \A}$ then $a_{\alpha \beta} > 0$ for any $\alpha, \beta \in \A$, we have
\[ \min_{\alpha, \beta \in \A}\frac{h^{(N)}_\alpha }{h^{(N)}_\beta} = \min_{\alpha, \beta \in \A}\frac{\sum_{\delta \neq \alpha} a_{\delta \alpha}}{\sum_{\delta \neq \alpha} a_{\delta \beta}} > c_\gamma,\]
 for some $c_\gamma > 0$ depending only on $\gamma$. 
\end{proof}

Since $c_\gamma$ in the previous claim does not depend on $c_0$, we assume WLOG that $c_\gamma> c_0$, by taking $c_0$ smaller if necessary. Since the initial permutation in $\gamma'$ is of top type and the matrix $A_{\gamma'}$ is positive, it follows that $\overline{\Delta_{\gamma'}} \subset \Delta_d^0$. In particular $A_n \subset \Delta_{\gamma'} \subset \overline{\Delta_{\gamma'}} \subset \Delta_d^0$. 

Therefore, since $\mu_\cZ$ is equivalent to the Lebesgue measure in $\Delta_d$ and its density is uniformly bounded away from the boundary of $\Delta_d$ (see Section \ref{sc: Zorich_map}), there exists $c_2 > 0$, depending only on $\gamma$, such that $$\mu(A_n) \geq c_2|A_n|.$$ Moreover, by definition of $A_n$ and (\ref{eq: lower_bound_set}),
 $$|A_n| \geq c_3 |\widehat{A}_n| \geq \frac{c_1c_3}{nC(n)},$$
for some $c_3 > 0$ depending only on $\gamma$. Hence
\[ \sum_{n \geq 1} \mu(A_n) = +\infty = \sum_{n \geq 1} \mu(A_{2n + N}), \]
since $C(n)$ is an increasing sequence verifying (\ref{eq: conditions_increasing_sequence}). 

Noticing that, for any $n \geq 1$, the sets $A_n$ are subsimplexes of $\Delta_{\gamma'} \in \mathcal{P}_{\mathfrak{R}}^N$, it follows by Theorem \ref{thm: borel_cantelli} that the set 
\[ \bigcap_{m \geq 1}\bigcup_{n \geq m} \mathcal{Z}^{-2n }(A_{2n + N})\]
has full measure in $\Delta_\mathfrak{R}^0$. Thus, for a.e. $(\lambda, \pi) \in \Delta_\mathfrak{R}^0$, there exists an increasing sequence $\{n_k\}_{k \geq 1} \subset \N$ such that
\[ (\lambda^{(n_k - N)}, \pi^{(n_k - N)}) \in A_{n_k},\]
for all $k \geq 1$. In particular, it follows from the claim that $\pi^{(n_k)} = \pi^*$ and $\lambda^{(n_k)}, h^{(n_k)}$ verify (\ref{eq: length_estimates}), (\ref{eq: height_estimates}), when replacing $N$ by $n_k$, for any $k \geq 1$. 

By considering the set 
\[ \mathcal{Z}^{-1}\left(\bigcap_{m \geq 1}\bigcup_{n \geq m} \mathcal{Z}^{-2n }(A_{2n + N + 1})\right)\]
and recalling that $\mathcal{Z}^{-1}(\Delta^0_\mathfrak{R}) = \Delta^1_\mathfrak{R}$ (up to a zero measure set), it is easy to show that the same conclusions hold for a.e. $(\lambda, \pi) \in \Delta_\mathfrak{R}^1$. This finishes the proof. 
\end{proof}

\bibliographystyle{acm}
\bibliography{Bibliography.bib}
\end{document}